\documentclass[a4paper]{amsart}
\usepackage{amscd}
\usepackage{amsmath}
\usepackage{amssymb}
\usepackage{mathrsfs}
\usepackage{amsthm}
\usepackage{bbm} 
\usepackage{stmaryrd}
\usepackage{tikz-cd}

\usepackage[T1]{fontenc}

\makeatletter
\@namedef{subjclassname@2020}{%
  \textup{2020} Mathematics Subject Classification}
\makeatother

\numberwithin{equation}{section} \swapnumbers

\newtheorem{satz}{Satz}[section]

\newtheorem{theorem}[satz]{Theorem}
\newtheorem{proposition}[satz]{Proposition}

\newtheorem{lemma}[satz]{Lemma}
\newtheorem{assumption}[satz]{Assumption}

\newtheorem{remark}[satz]{Remark}

\newcommand{\bbr}{\mathbb{R}}
\newcommand{\bbe}{\mathbb{E}}
\newcommand{\bbn}{\mathbb{N}}
\newcommand{\bbp}{\mathbb{P}}

\newcommand{\calb}{\mathscr{B}}

\newcommand{\calf}{\mathscr{F}}

\newcommand{\calh}{\mathscr{H}}

\newcommand{\calm}{\mathscr{M}}

\newcommand{\loc}{{\rm loc}}

\newcommand{\la}{\langle}
\newcommand{\ra}{\rangle}

\newcommand{\bbI}{\mathbbm{1}}

\begin{document}

\title[Mild solutions to stochastic partial differential equations]{Mild solutions to semilinear stochastic partial differential equations with locally monotone coefficients}
\author{Stefan Tappe}
\address{Albert Ludwig University of Freiburg, Department of Mathematical Stochastics, Ernst-Zermelo-Stra\ss{}e 1, D-79104 Freiburg, Germany}
\email{stefan.tappe@math.uni-freiburg.de}
\date{21 April, 2021}
\thanks{I gratefully acknowledge financial support from the Deutsche Forschungsgemeinschaft (DFG, German Research Foundation) -- project number 444121509.}
\begin{abstract}
We provide an existence and uniqueness result for mild solutions to semilinear stochastic partial differential equations in the framework of the semigroup approach with locally monotone coefficients. An important component of the proof is an application of the dilation theorem of Nagy, which allows us to reduce the problem to infinite dimensional stochastic differential equations on a larger Hilbert space. Properties of the solutions like the Markov property are discussed as well.
\end{abstract}

\maketitle\thispagestyle{empty}

\section{Introduction}

In this article we consider semilinear stochastic partial differential equations (SPDEs) in the framework of the semigroup approach driven by a cylindrical Wiener process; see, for example \cite{Da_Prato, Atma-book}. It is well-known that existence and uniqueness of mild solutions holds true for such SPDEs if the coefficients satisfy Lipschitz type conditions and the volatility only depends on the space variable; see, for example \cite[Thm. 7.5]{Da_Prato}. However, to the best of my knowledge, in the literature there is no existence and uniqueness result for mild solutions to SPDEs under local monotonicity and coercivity conditions on the coefficients. The goal of the present paper is to fill this gap and to present such a result; see Theorem \ref{thm-SPDE} below. Moreover, in contrast to the just cited result, the coefficients may depend on space, time and randomness.

We point out that existence and uniqueness of strong solutions under monotonicity and coercivity conditions is well-known in the framework of the variational approach; see, for example \cite{KR, RRW, Liu, LR-2010, LR-2013}, the textbooks \cite{Prevot-Roeckner, Liu-Roeckner} and the recent article \cite{Marinelli}, where an alternative proof is presented.

The essential idea for proving our existence and uniqueness result is to utilize the ``method of the moving frame''. This method has originally been presented in \cite{SPDE}; it provides a link between mild solutions to SPDEs and strong solutions to infinite dimensional stochastic differential equations (SDEs) on a larger Hilbert space, which ensures that we can use the results in the framework of the variational approach. This method relies on the dilation theorem of Nagy, which allows us to extend the semigroup to a group on the larger Hilbert space; see, for example \cite{Nagy, Davies}. Besides the aforementioned article \cite{SPDE}, the ``method of the moving frame'' has also been used in \cite{Tappe-refine} for certain existence and uniqueness results, and in \cite{Tappe-YW, Tappe-YW-dual} for Yamada-Watanabe type results regarding SPDEs. In the present paper, we will use it in order to establish our existence and uniqueness result, and also in order to obtain a short proof of the Markov property of the solutions.

The remainder of this paper is organized as follows. In Section \ref{sec-SPDE} we provide the announced existence and uniqueness result for SPDEs, and in Section \ref{sec-Markov} we prove the Markov property of solutions for non-random coefficients. For convenience of the reader, we provide the required results about infinite dimensional SDEs in Appendix \ref{sec-SDE}.

\section{Existence and uniqueness result}\label{sec-SPDE}

In this section, we provide the announced existence and uniqueness result for semilinear SPDEs with locally monotone coefficients. Let $T > 0$ be a finite time horizon, and let $(\Omega,\calf,(\calf_t)_{t \in [0,T]},\bbp)$ be a filtered probability space satisfying the usual conditions. Let $H$ and $U$ be separable Hilbert spaces, let $A$ be the generator of a $C_0$-semigroup $(S_t)_{t \geq 0}$ of contractions on $H$, and let $W = (W_t)_{t \in [0,T]}$ be a cylindrical Wiener process in $U$. We consider the $H$-valued SPDE
\begin{align}\label{SPDE}
dX_t = (A X_t + \alpha(t,X_t))dt + \sigma(t,X_t) dW_t
\end{align}
with progressively measurable coefficients
\begin{align*}
\alpha : [0,T] \times H \times \Omega \to H \quad \text{and} \quad \sigma : [0,T] \times H \times \Omega \to L_2(U,H),
\end{align*}
where $L_2(U,H)$ denotes the space of all Hilbert-Schmidt operators from $U$ to $H$. Given an $\calf_0$-measurable random variable $\xi : \Omega \to H$, an $H$-valued continuous, adapted process $X = (X_t)_{t \in [0,T]}$ is called a \emph{mild solution} to the SPDE (\ref{SPDE}) with $X_0 = \xi$ if we have $\bbp$-almost surely
\begin{align*}
\int_0^T \big( \| \alpha(s,X_s) \|_H + \| \sigma(s,X_s) \|_{L_2(U,H)}^2 \big) ds < \infty
\end{align*}
as well as
\begin{align*}
X_t = S_t \xi + \int_0^t S_{t-s} \alpha(s,X_s) ds + \int_0^t S_{t-s} \sigma(s,X_s) dW_s, \quad t \in [0,T].
\end{align*}
The following dilation of the semigroup $(S_t)_{t \geq 0}$ will be crucial for our analysis of the existence of mild solutions to the SPDE (\ref{SPDE}).

\begin{theorem}\label{thm-diagram}
There exist another separable Hilbert space $\calh$, a unitary $C_0$-group $(U_t)_{t \in \mathbb{R}}$ on $\calh$ and an isometric embedding $\ell \in L(H,\calh)$ such that the diagram
\[ \begin{CD}
\calh @>U_t>> \calh\\
@AA\ell A @VV\pi V\\
H @>S_t>> H
\end{CD} \]
commutes for every $t \in \mathbb{R}_+$, that is
\begin{align}\label{diagram-commutes}
\pi U_t \ell = S_t \quad \text{for all $t \in \mathbb{R}_+$,}
\end{align}
where $\pi := \ell^*$ is the orthogonal projection from $\calh$ into $H$.
\end{theorem}

\begin{proof}
Since $(S_t)_{t \geq 0}$ is a semigroup of contractions, this is an immediate consequence of \cite[Thm. I.8.1]{Nagy}.
\end{proof}

\begin{remark}\label{remark-group}
Since the group $(U_t)_{t \in \mathbb{R}}$ is unitary, we have $U_{-t} = U_t^*$ and $U_t$ is an isometry for all $t \in \bbr$.
\end{remark}

In order to establish our existence result, we will also consider the $\calh$-valued SDE (\ref{SDE}) with progressively measurable coefficients
\begin{align*}
a : [0,T] \times \calh \times \Omega \to \calh \quad \text{and} \quad b : [0,T] \times \calh \times \Omega \to L_2(U,\calh)
\end{align*}
given by
\begin{align}\label{def-ab}
a(t,v,\omega) := U_t^* \ell \alpha(t,\pi U_t v,\omega) \quad \text{and} \quad b(t,v,\omega) := U_t^* \ell \sigma(t,\pi U_t v,\omega).
\end{align}
For a continuous function $h : [0,T] \to \calh$ we define the continuous function $\pi U h : [0,T] \to H$ as
\begin{align*}
(\pi U h)(t) := \pi U_t h(t), \quad t \in [0,T].
\end{align*}
The following two auxiliary results show how mild solutions to the SPDE (\ref{SPDE}) and strong solutions to the SDE (\ref{SDE}) are related.

\begin{lemma}\cite[Cor. 3.9]{Tappe-YW}\label{lemma-trans-1}
Let $\xi : \Omega \to H$ be an $\calf_0$-measurable random variable, let $X$ be a mild solution to the SPDE (\ref{SPDE}) with $X_0 = \xi$, and set
\begin{align*}
Y := \ell \xi + \int_0^{\bullet} a(s,X_s) ds + \int_0^{\bullet} b(s,X_s) dW_s.
\end{align*}
Then the following statements are true:
\begin{enumerate}
\item $Y$ is a strong solution to the SDE (\ref{SDE}) with $Y_0 = \ell \xi$.

\item We have $\bbp$-almost surely $X = \pi U Y$.
\end{enumerate}
\end{lemma}

\begin{lemma}\cite[Cor. 3.11]{Tappe-YW}\label{lemma-trans-2}
Let $\xi : \Omega \to H$ be an $\calf_0$-measurable random variable, and let $Y$ be a strong solution to the SDE (\ref{SDE}) with $Y_0 = \ell \xi$. Then the following statements are true:
\begin{enumerate}
\item $X := \pi U Y$ is a mild solution to the SPDE (\ref{SPDE}) with $X_0 = \xi$.

\item We have $\bbp$-almost surely
\begin{align*}
Y = \ell \xi + \int_0^{\bullet} a(s,X_s) ds + \int_0^{\bullet} b(s,X_s) dW_s.
\end{align*}
\end{enumerate}
\end{lemma}

\begin{assumption}\label{ass-SPDE}
We assume there are constants $\beta, C_0, \theta \in \bbr_+$ such that $C_0 \geq \theta > 0$ and a nonnegative adapted process $f \in L^1(\Omega \times [0,T]; dt \otimes \bbp)$ such that for all $x,y,z \in H$ and all $(t,\omega) \in [0,T] \times \Omega$ the following conditions are fulfilled:
\begin{enumerate}
\item[(H1)] (Hemicontinuity) The map $\lambda \mapsto \la \alpha(t,x+\lambda y,\omega),z \ra_H$ is continuous on $\bbr$.

\item[(H2')] (Local monotonicity) We have
\begin{align*}
&2\la \alpha(t,x,\omega) - \alpha(t,y,\omega), x-y \ra_H + \| \sigma(t,x,\omega) - \sigma(t,y,\omega) \|_{L_2(U,H)}^2
\\ &\leq (f(t,\omega) + \tau(\| y \|_H)) \| x-y \|_H^2,
\end{align*}
where $\tau : \bbr_+ \to \bbr_+$ is a continuous, increasing function.

\item[(H3)] (Coercivity) We have
\begin{align*}
2 \la \alpha(t,y,\omega),y \ra_H + \| \sigma(t,y,\omega) \|_{L_2(U,H)}^2 \leq (C_0-\theta) \| y \|_H^2 + f(t,\omega).
\end{align*}
\item[(H4')] (Growth) We have
\begin{align*}
\| \alpha(t,y,\omega) \|_H^2 \leq ( f(t,\omega) + C_0 \| y \|_H^2 ) (1 + \| y \|_H^{\beta}).
\end{align*}
\end{enumerate}
\end{assumption}

Here is our result concerning existence and uniqueness of mild solutions to the SPDE (\ref{SPDE}).

\begin{theorem}\label{thm-SPDE}
Suppose that Assumption \ref{ass-SPDE} is satisfied for some $f \in L^{p/2}([0,T] \times \Omega; dt \otimes \bbp)$ with some $p \geq \beta + 2$, and that there is a constant $C \in \bbr_+$ such that
\begin{align}\label{sigma-zusatz-bed}
\| \sigma(t,y,\omega) \|_{L_2(U,H)}^2 &\leq C ( f(t,\omega) + \| y \|_H^2 ), \quad (t,y,\omega) \in [0,T] \times H \times \Omega,
\\ \tau(r) &\leq C(1 + r^2) (1 + r^{\beta}), \quad r \in \bbr_+.
\end{align}
Then the following statements are true:
\begin{enumerate}
\item For each $\calf_0$-measurable random variable $\xi : \Omega \to \calh$ there is a unique mild solution $X$ to the SPDE (\ref{SPDE}) with $X_0 = \xi$.

\item If $\xi \in L^p(\Omega,\calf_0,\bbp;H)$, then we have
\begin{align*}
\bbe \bigg[ \sup_{t \in [0,T]} \| X_t \|_H^p \bigg] < \infty.
\end{align*}

\item There is an increasing function $K : \bbr_+ \to \bbr_+$ such that for every $\xi \in L^2(\Omega,\calf_0,\bbp;H)$ we have
\begin{align*}
\bbe \big[ \| X_t \|_H^2 \big] \leq K(t) \big( 1 + \bbe[\| \xi \|_H^2] \big), \quad t \in [0,T],
\end{align*}
where $X$ denotes the mild solution to the SPDE (\ref{SPDE}) with $X_0 = \xi$.

\item If $f \in L^{p/2}([0,T])$ is deterministic and $\tau \in \bbr_+$ is a constant, then for each $t \in [0,T]$ the solution map
\begin{align*}
L^2(\Omega,\calf_0,\bbp;H) \to L^2(\Omega,\calf_t,\bbp;H), \quad \xi \mapsto X_t(\xi),
\end{align*}
where $X(\xi)$ denotes the strong solution to the SPDE (\ref{SPDE}) with $X_0(\xi) = \xi$, is Lipschitz continuous.
\end{enumerate}
\end{theorem}

\begin{proof}
Using identity (\ref{diagram-commutes}), Remark \ref{remark-group}, the definitions (\ref{def-ab}) and Assumption \ref{ass-SPDE}, we will verify that Assumption \ref{ass-SDE} is fulfilled. Let $u,v,w \in \calh$ and $(t,\omega) \in [0,T] \times \Omega$ be arbitrary.
\begin{enumerate}
\item[(H1)] (Hemicontinuity) The map
\begin{align*}
\lambda \mapsto \la a(t,u+\lambda v,\omega),w \ra_{\calh} &= \la U_t^* \ell \alpha(t,\pi U_t(u+\lambda v),\omega),w \ra_{\calh}
\\ &= \la \alpha(t,\pi U_t u + \lambda \pi U_t v,\omega), \pi U_t w \ra_H
\end{align*}
is continuous on $\bbr$. 

\item[(H2')] (Local monotonicity) We have
\begin{align*}
&2 \la a(t,u,\omega) - a(t,v,\omega), u-v \ra_{\calh} + \| b(t,u,\omega) - b(t,v,\omega) \|_{L_2(U,\calh)}^2
\\ &=2 \la U_t^* \ell \alpha(t,\pi U_t u,\omega) - U_t^* \ell \alpha(t,\pi U_t v,\omega), u-v \ra_{\calh}
\\ &\quad + \| U_t^* \ell \sigma(t, \pi U_t u,\omega) - U_t^* \ell \sigma(t,\pi U_t v,\omega) \|_{\calh}^2
\\ &= 2 \la \alpha(t,\pi U_t u,\omega) - \alpha(t,\pi U_t v,\omega), \pi U_t(u-v) \ra_H + \| \sigma(t,\pi U_t u,\omega) - \sigma(t,\pi U_t v,\omega) \|_H^2
\\ &\leq (f(t,\omega) + \tau(\| \pi U_t v \|_H) \| \pi U_t u - \pi U_t v \|_H^2
\\ &\leq (f(t,\omega) + \tau(\| v \|_{\calh}) \| u - v \|_{\calh}^2.
\end{align*}

\item[(H3)] (Coercivity) We have
\begin{align*}
&2 \la a(t,v,\omega),v \ra_{\calh} + \| b(t,v,\omega) \|_{L_2(U,\calh)}^2
\\ &= 2 \la U_t^* \ell \alpha(t,\pi U_t v,\omega), v \ra_{\calh} + \| U_t^* \ell \sigma(t,\pi U_t v,\omega) \|_{L_2(U,\calh)}^2
\\ &= 2 \la \alpha(t,\pi U_t v,\omega), \pi U_t v \ra_H + \| \sigma(t,\pi U_t v,\omega) \|_{L_2(U,H)}^2
\\ &\leq (C_0 - \theta) \| \pi U_t v \|_H^2 + f(t,\omega)
\\ &\leq (C_0 - \theta) \| v \|_{\calh}^2 + f(t,\omega).
\end{align*}

\item[(H4')] (Growth) We have
\begin{align*}
\| a(t,v,\omega) \|_{\calh}^2 &= \| U_t^* \ell \alpha(t,\pi U_t v,\omega) \|_{\calh}^2 = \| \alpha(t,\pi U_t v,\omega) \|_H^2
\\ &\leq (f(t,\omega) + C_0 \| \pi U_t v \|_H^2) (1 + \| \pi U_t v \|_H^{\beta})
\\ &\leq (f(t,\omega) + C_0 \| v \|_{\calh}^2) (1 + \| v \|_{\calh}^{\beta}).
\end{align*}
\end{enumerate}
Furthermore, by (\ref{sigma-zusatz-bed}) we have
\begin{align*}
\| b(t,v,\omega) \|_{L_2(U,\calh)}^2 &= \| U_t^* \ell \sigma(t,\pi U_t v,\omega) \|_{L_2(U,\calh)}^2 = \| \sigma(t,\pi U_t v,\omega) \|_{L_2(U,H)}^2
\\ &\leq C(f(t,\omega) + \| \pi U_t v \|_H^2) \leq C(f(t,\omega) + \| v \|_{\calh}^2),
\end{align*}
showing (\ref{b-zusatzbedingung}). Consequently, all assumptions from Theorem \ref{thm-SDE} are fulfilled. Therefore, noting Remark \ref{remark-group} and Lemmas \ref{lemma-trans-1}, \ref{lemma-trans-2}, all statements immediately follow.
\end{proof}

\section{Markov property of solutions}\label{sec-Markov}

In this section, we establish the Markov property of solutions for non-random coefficients. More precisely, we consider the framework from Section \ref{sec-SPDE} and assume that the coefficients do not depend on $\omega \in \Omega$; that is, we have measurable mappings
\begin{align*}
\alpha : [0,T] \times H \to H \quad \text{and} \quad b : [0,T] \times H \to L_2(U,H).
\end{align*}
Furthermore, we assume that Assumption \ref{ass-SPDE} is satisfied for some deterministic $f \in L^{p/2}([0,T])$ with some $p \geq \beta + 2$, that $\tau \in \bbr_+$ is a constant, and that there is a constant $C \in \bbr_+$ such that
\begin{align}
\| \sigma(t,y) \|_{L_2(U,H)}^2 &\leq C ( f(t) + \| y \|_H^2 ), \quad (t,y) \in [0,T] \times H.
\end{align}
According to Theorem \ref{thm-SPDE}, for each $s \in [0,T]$ and each $\calf_s$-measurable random variable $\xi : \Omega \to H$ there exists a unique solution $X = X(s,\xi)$ to the equation
\begin{align*}
X_t = S_{t-s} \xi + \int_s^t S_{t-u} \alpha(u,X_u) du + \int_s^t S_{t-u} \sigma(u,X_u) dW_u, \quad t \in [s,T]
\end{align*}
with underlying Wiener process $W_t - W_s$, $t \in [s,T]$ and filtration $(\calf_t)_{t \geq s}$. Using the uniqueness of solutions, for all $0 \leq r \leq s \leq t \leq T$ and every $\calf_r$-measurable random variable $\xi : \Omega \to H$ we have the flow property
\begin{align}\label{flow-property}
X_t(r,\xi) = X_t(s,X_s(r,\xi)).
\end{align}
For the next auxiliary result, recall the $\calh$-valued SDE (\ref{SDE}) with coefficients
\begin{align*}
a : [0,T] \times \calh \to \calh \quad \text{and} \quad b : [0,T] \times \calh \to L_2(U,\calh)
\end{align*}
defined according to (\ref{def-ab}).

\begin{lemma}\label{lemma-X-Y-start-in-t}
For all $s \in [0,T]$ and every $\calf_s$-measurable random variable $\xi : \Omega \to H$ we have
\begin{align*}
X_t(s,\xi) = \pi U_t Y_t(s, U_s^* \ell \xi), \quad t \in [s,T].
\end{align*}
\end{lemma}

\begin{proof}
Using the notation $Y = Y(s,U_s^* \ell \xi)$, by (\ref{diagram-commutes}) for all $t \in [s,T]$ we obtain
\begin{align*}
\pi U_{t} Y_t &= \pi U_{t} \bigg( U_s^* \ell \xi + \int_s^t a(u,Y_u) du + \int_s^t b(u,Y_u) dW_u \bigg)
\\ &= S_{t-s} \xi + \pi U_{t} \int_s^t U_{u}^* \ell \alpha(u,\pi U_u Y_u) du + \pi U_{t} \int_s^t U_{u}^* \ell \sigma(u,\pi U_u Y_u) d W_u
\\ &= S_{t-s} \xi + \int_s^t S_{t-u} \alpha(u,\pi U_u Y_u) du + \int_s^t S_{t-u} \sigma(u,\pi U_u Y_u) dW_u.
\end{align*}
By the uniqueness of solutions, this completes the proof.
\end{proof}

We denote by $B_b(H)$ the Banach space of all bounded Borel functions $\varphi : H \to \bbr$, endowed with the supremum norm. Similarly, we denote by $C_b(H)$ the Banach space of all bounded and continuous functions $\varphi : H \to \bbr$. For $0 \leq s \leq t \leq T$ we denote by $P_{s,t} \in L(B_b(H))$ the bounded linear operator defined as
\begin{align*}
(P_{s,t} \varphi)(x) := \bbe[\varphi(X_t(s,x))], \quad x \in H
\end{align*}
for each $\varphi \in B_b(H)$.

\begin{theorem}[Markov property]\label{thm-Markov-prop}
For all $0 \leq r \leq s \leq t \leq T$, every $\calf_r$-measurable random variable $\xi : \Omega \to H$ and every $\varphi \in B_b(H)$ we have $\bbp$-almost surely
\begin{align*}
\bbe[\varphi(X_t(r,\xi)) | \calf_s] = (P_{s,t} \varphi)(X_s(r,\xi)).
\end{align*}
\end{theorem}

\begin{proof}
Defining the $\calf_s$-measurable random variable $\vartheta := X_s(r,\xi)$, by the flow property (\ref{flow-property}) we have to show that $\bbp$-almost surely
\begin{align*}
\bbe[\varphi(X_t(s,\vartheta)) | \calf_s] = \bbe[\varphi(X_t(s,x))]|_{x = \vartheta}.
\end{align*}
Indeed, noting that $\varphi \circ \pi \circ U_t \in B_b(\calh)$, by Lemma \ref{lemma-X-Y-start-in-t} and Proposition \ref{prop-Markov-SDE} we obtain $\bbp$-almost surely
\begin{align*}
\bbe[\varphi(X_t(s,\vartheta)) | \calf_s] &= \bbe[\varphi(\pi U_t Y_t(s,U_s^* \ell \vartheta))|\calf_s]
\\ &= \bbe[\varphi(\pi U_t Y_t(s,y))]|_{y = U_s^* \ell \vartheta}
\\ &= \bbe[\varphi(\pi U_t Y_t(s,U_s^* \ell x))]|_{x = \vartheta}
\\ &= \bbe[\varphi(X_t(s,x))]|_{x = \vartheta},
\end{align*}
completing the proof.
\end{proof}

\begin{remark}
We can provide an alternative proof of Theorem \ref{thm-Markov-prop} by following the proof of \cite[Prop. 4.3.3]{Liu-Roeckner} and using the Yamada-Watanabe theorem for mild solutions to semilinear SPDEs, which has been presented in \cite{Tappe-YW}.
\end{remark}

Consequently, for all $0 \leq r \leq s \leq t \leq T$ we have $P_{r,s} P_{s,t} = P_{r,t}$. Furthermore, defining for $0 \leq s \leq t \leq T$ and $x \in H$ the transition function $P(s,x;t,\cdot)$ as the probability measure
\begin{align*}
P(s,x;t,\Gamma) := (P_{s,t} \bbI_{\Gamma})(x), \quad \Gamma \in \calb(H),
\end{align*}
we deduce the Chapman-Kolmogorov equation
\begin{align*}
P(r,x;t,\Gamma) = \int_H P(r,x;s,dy) P(s,y;t,\Gamma), \quad \Gamma \in \calb(H)
\end{align*}
for all $0 \leq r \leq s \leq t \leq T$ and $x \in H$.

\begin{appendix}

\section{Infinite dimensional stochastic differential equations}\label{sec-SDE}

In this appendix, we provide the required results about infinite dimensional SDEs. Let $T > 0$ be a finite time horizon, and let $(\Omega,\calf,(\calf_t)_{t \in [0,T]},\bbp)$ be a filtered probability space satisfying the usual conditions. Let $\calh$ and $U$ be separable Hilbert spaces, and let $W = (W_t)_{t \in [0,T]}$ be a cylindrical Wiener process in $U$. We consider the $\calh$-valued SDE
\begin{align}\label{SDE}
dY_t = a(t,Y_t) dt + b(t,Y_t)dW_t
\end{align}
with progressively measurable coefficients
\begin{align*}
a : [0,T] \times \calh \times \Omega \to \calh \quad \text{and} \quad b : [0,T] \times \calh \times \Omega \to L_2(U,\calh).
\end{align*}
Given an $\calf_0$-measurable random variable $\eta : \Omega \to \calh$, an $\calh$-valued continuous, adapted process $Y = (Y_t)_{t \in [0,T]}$ is called a \emph{strong solution} to the SDE (\ref{SDE}) with $Y_0 = \eta$ if we have $\bbp$-almost surely
\begin{align*}
\int_0^T \big( \| a(s,Y_s) \|_{\calh} + \| b(s,Y_s) \|_{L_2(U,\calh)}^2 \big) ds < \infty
\end{align*}
as well as
\begin{align*}
Y_t = \eta + \int_0^t a(s,Y_s) ds + \int_0^t b(s,Y_s) dW_s, \quad t \in [0,T].
\end{align*}

\begin{assumption}\label{ass-SDE}
We assume there are constants $\beta,C_0,\theta \in \bbr_+$ such that $C_0 \geq \theta > 0$ and a nonnegative adapted process $f \in L^1([0,T] \times \Omega; dt \otimes \bbp)$ such that for all $u,v,w \in \calh$ and all $(t,\omega) \in [0,T] \times \Omega$ the following conditions are fulfilled:
\begin{enumerate}
\item[(H1)] (Hemicontinuity) The map $\lambda \mapsto \la a(t,u+\lambda v,\omega),w \ra_{\calh}$ is continuous on $\bbr$.

\item[(H2')] (Local monotonicity) We have
\begin{align*}
&2\la a(t,u,\omega) - a(t,v,\omega), u-v \ra_{\calh} + \| b(t,u,\omega) - b(t,v,\omega) \|_{L_2(U,\calh)}^2
\\ &\leq (f(t,\omega) + \tau(\| v \|_{\calh})) \| u-v \|_{\calh}^2,
\end{align*}
where $\tau : \bbr_+ \to \bbr_+$ is a continuous, increasing function.

\item[(H3)] (Coercivity) We have
\begin{align*}
2 \la a(t,v,\omega),v \ra_{\calh} + \| b(t,v,\omega) \|_{L_2(U,\calh)}^2 \leq (C_0-\theta) \| v \|_{\calh}^2 + f(t,\omega).
\end{align*}
\item[(H4')] (Growth) We have
\begin{align*}
\| a(t,v,\omega) \|_{\calh}^2 \leq ( f(t,\omega) + C_0 \| v \|_{\calh}^2 ) (1 + \| v \|_{\calh}^{\beta}).
\end{align*}
\end{enumerate}
\end{assumption}

\begin{remark}
Note that Assumption \ref{ass-SDE} corresponds to conditions (H1), (H2'), (H3), (H4') in \cite[Sec. 5.1.1]{Liu-Roeckner} with $V = H = \calh$, $\alpha = 2$, $C_0 \geq \theta$ and the function $\rho : \calh \to \bbr_+$ being of the particular form $\rho(v) = \tau(\| v \|_{\calh})$ for $v \in \calh$. These slightly stronger conditions allow us to transfer the upcoming result for SDEs (see Theorem \ref{thm-SDE}) to the framework of mild solutions for SPDEs (see Theorem \ref{thm-SPDE}) by using the method of the moving frame.
\end{remark}

\begin{theorem}\label{thm-SDE}
Suppose that Assumption \ref{ass-SDE} is satisfied for some $f \in L^{p/2}([0,T] \times \Omega; dt \otimes \bbp)$ with some $p \geq \beta + 2$, and there is a constant $C \in \bbr_+$ such that
\begin{align}\label{b-zusatzbedingung}
\| b(t,v,\omega) \|_{L_2(U,\calh)}^2 &\leq C ( f(t,\omega) + \| v \|_{\calh}^2 ), \quad (t,v,\omega) \in [0,T] \times \calh \times \Omega,
\\ \tau(r) &\leq C(1 + r^2) (1 + r^{\beta}), \quad r \in \bbr_+.
\end{align}
Then the following statements are true:
\begin{enumerate}
\item For each $\calf_0$-measurable random variable $\eta : \Omega \to \calh$ there exists a unique strong solution $Y$ to the SDE (\ref{SDE}) with $Y_0 = \eta$.

\item If $\eta \in L^p(\Omega,\calf_0,\bbp;\calh)$, then we have
\begin{align*}
\bbe \bigg[ \sup_{t \in [0,T]} \| Y_t \|_{\calh}^p \bigg] < \infty.
\end{align*}

\item There is an increasing function $K : \bbr_+ \to \bbr_+$ such that for every $\eta \in L^2(\Omega,\calf_0,\bbp;\calh)$ we have
\begin{align*}
\bbe \big[ \| Y_t \|_{\calh}^2 \big] \leq K(t) \big( 1 + \bbe[\| \eta \|_{\calh}^2] \big), \quad t \in [0,T],
\end{align*}
where $Y$ denotes the strong solution to the SDE (\ref{SDE}) with $Y_0 = \eta$.

\item If $f \in L^{p/2}([0,T])$ is deterministic and $\tau \in \bbr_+$ is a constant, then for each $t \in [0,T]$ the solution map
\begin{align*}
L^2(\Omega,\calf_0,\bbp;\calh) \to L^2(\Omega,\calf_t,\bbp;\calh), \quad \eta \mapsto Y_t(\eta),
\end{align*}
where $Y(\eta)$ denotes the strong solution to the SDE (\ref{SDE}) with $Y_0(\eta) = \eta$, is Lipschitz continuous.
\end{enumerate}
\end{theorem}

\begin{proof}
If $\eta \in L^p(\Omega,\calf_0,\bbp;\calh)$, then the first two statements are immediate consequences of \cite[Thm. 5.1.3]{Liu-Roeckner}. Now, for an arbitrary $\calf_0$-measurable random variable $\eta : \Omega \to \calh$, the first statement follows by considering the $\calf_0$-measurable partition $(\Omega_n)_{n \in \bbn}$ of $\Omega$ given by
\begin{align*}
\Omega_n := \{ \| \eta \|_{\calh} \in [n-1,n) \}
\end{align*}
and the sequence $(\eta_n)_{n \in \bbn}$ of $\calf_0$-measurable random variables $\eta_n := \eta \bbI_{\Omega_n}$ for each $n \in \bbn$. Now, let $t \in [0,T]$ be arbitrary. By It\^{o}'s formula (see, for example \cite[Thm. 4.2.5]{Liu-Roeckner}) and the coercivity condition (H3) we obtain
\begin{align*}
\| Y_t \|_{\calh}^2 e^{-(C_0 - \theta)t} &= \| \eta \|_{\calh}^2 + \int_0^t e^{-(C_0 - \theta)s} \Big( 2 \la a(s,Y_s), Y_s \ra_{\calh} + \| b(s,Y_s) \|_{L_2(U,\calh)}^2
\\ &\qquad\qquad\qquad - (C_0 - \theta) \| Y_s \|_{\calh}^2 \Big) ds + M_t
\\ &\leq \| \eta \|_{\calh}^2 + \int_0^t e^{-(C_0 - \theta)s} f(s) ds + M_t
\\ &\leq \| \eta \|_{\calh}^2 + \int_0^T f(s) ds + M_t
\end{align*}
with a continuous local martingale $M \in \calm_{\loc}$ such that $M_0 = 0$. Therefore, by localization and using Fatou's lemma we obtain
\begin{align*}
\bbe \big[ \| Y_t \|_{\calh}^2 e^{-(C_0 - \theta)t} \big] \leq \bbe[ \| \eta \|_{\calh}^2 ] + \| f \|_{L^1},
\end{align*}
proving the claimed estimate. Finally, if $f \in L^{p/2}([0,T])$ is deterministic and $\tau \in \bbr_+$ is a constant, then the proof of the aforementioned result (see \cite[p. 132]{Liu-Roeckner}) shows that for two strong solutions $Y$ and $Z$ to the SDE (\ref{SDE}) we have
\begin{align*}
\bbe \bigg[ \exp \bigg( -\int_0^t f(s) ds - \tau t \bigg) \| Y_t - Z_t \|_{\calh}^2 \bigg] \leq \bbe \big[ \| Y_0 - Z_0 \|_{\calh}^2 \big], \quad t \in [0,T],
\end{align*}
providing the asserted Lipschitz continuity.
\end{proof}

For the rest of this section, we assume that the coefficients do not depend on $\omega \in \Omega$; that is, we have measurable mappings
\begin{align*}
a : [0,T] \times \calh \to \calh \quad \text{and} \quad b : [0,T] \times \calh \to L_2(U,\calh).
\end{align*}
Furthermore, we assume that Assumption \ref{ass-SDE} is satisfied for some deterministic $f \in L^{p/2}([0,T])$ with some $p \geq \beta + 2$, that $\tau \in \bbr_+$ is a constant, and that there is a constant $C \in \bbr_+$ such that
\begin{align}
\| b(t,v) \|_{L_2(U,\calh)}^2 &\leq C ( f(t) + \| v \|_{\calh}^2 ), \quad (t,v) \in [0,T] \times \calh.
\end{align}
According to Theorem \ref{thm-SDE}, for each $s \in [0,T]$ and each $\calf_s$-measurable random variable $\xi : \Omega \to H$ there exists a unique solution $Y = Y(s,\eta)$ to the equation
\begin{align*}
Y_t = \eta + \int_s^t a(u,Y_u) du + \int_s^t b(u,Y_u) dW_u, \quad t \in [s,T]
\end{align*}
with underlying Wiener process $W_t - W_s$, $t \in [s,T]$ and filtration $(\calf_t)_{t \geq s}$.

\begin{proposition}\label{prop-Markov-SDE}
For all $0 \leq s \leq t \leq T$, every $\calf_s$-measurable random variable $\eta : \Omega \to \calh$ and every $\Phi \in B_b(\calh)$ we have $\bbp$-almost surely
\begin{align}\label{eqn-Markov-Y}
\bbe[\Phi(Y_t(s,\eta)) | \calf_s] = \bbe[\Phi(Y_t(s,y))]|_{y = \eta}.
\end{align}
\end{proposition}

\begin{proof}
We provide an outline of the proof, which follows a well-known pattern. First of all, by a monotone class argument, we may assume that $\Phi \in C_b(\calh)$. Then we proceed with the following steps:
\begin{enumerate}
\item[(1)] If $\eta = y$ almost surely for some $y \in \calh$, then equation (\ref{eqn-Markov-Y}) is satisfied, because $Y(s,\eta)$ is independent of $\calf_s$.

\item[(2)] Then we show equation (\ref{eqn-Markov-Y}) for every $\calf_s$-measurable random variable $\eta : \Omega \to \calh$ of the form $\eta = \sum_{j=1}^N x_j \bbI_{\Gamma_j}$ with $N \in \bbn$, elements $x_1,\ldots,x_N \in \calh$ and an $\calf_s$-measurable partition $(\Gamma_j)_{j=1,\ldots,N}$ of $\Omega$.

\item[(3)] Using the Lipschitz continuity of the solution map
\begin{align*}
L^2(\Omega,\calf_0,\bbp;\calh) \to L^2(\Omega,\calf_t,\bbp;\calh), \quad \eta \mapsto Y_t(s,\eta)
\end{align*}
from Theorem \ref{thm-SDE} and the continuity of $\Phi$, we prove equation (\ref{eqn-Markov-Y}) for every $\calf_s$-measurable random variable $\eta : \Omega \to \calh$ such that $\bbe[\| \eta \|_{\calh}^2] < \infty$.

\item[(4)] For a general $\calf_s$-measurable random variable $\eta : \Omega \to \calh$ we show equation (\ref{eqn-Markov-Y}) by considering the $\calf_s$-measurable partition $(\Omega_n)_{n \in \bbn}$ of $\Omega$ given by
\begin{align*}
\Omega_n := \{ \| \eta \|_{\calh} \in [n-1,n) \}
\end{align*}
and the sequence $(\eta_n)_{n \in \bbn}$ of $\calf_s$-measurable random variables $\eta_n := \eta \bbI_{\Omega_n}$ for each $n \in \bbn$.
\end{enumerate}
\end{proof}

\end{appendix}

\end{document}